\documentclass[a4paper,12pt,oneside,bibliography=totoc]{article}

\usepackage{ucs}                     
\usepackage[utf8x]{inputenc}         
\usepackage[T1]{fontenc}             
\usepackage[english]{babel}          
\usepackage{amsmath, amssymb, amstext, bm}    
\usepackage[automark]{scrpage2}
\usepackage{lmodern}
\usepackage{textcomp}                
\usepackage{caption}                 
\usepackage{subcaption}              
\usepackage{multirow}
\usepackage{enumitem}

\usepackage{t1enc}
\usepackage{verbatim} 
\usepackage{graphicx}
\usepackage[dvips]{epsfig}
\usepackage{url}
\usepackage{titlesec}
\usepackage{pdfpages}
\usepackage[nottoc,notlot,notlof]{tocbibind}
\usepackage{hyperref}

\titleformat{\chapter}
{\normalfont\Large\bfseries}{\thechapter}{1em}{}
\titlespacing*{\chapter}{0pt}{3.5ex plus 1ex minus .2ex}{2.3ex plus .2ex}

\titleformat{\section}
{\normalfont\large\bfseries}{\thesection}{1em}{}

\titleformat{\subsection}
{\normalfont\normalsize\bfseries}{\thesection}{1em}{}

\usepackage[thmmarks,amsmath,hyperref,noconfig]{ntheorem} 
\newtheorem{thm}{Theorem} 
\newtheorem{lemma}{Lemma}

\newtheorem{thmm}{Theorem}

\theorembodyfont{\upshape}

\newtheorem{definition}{Definition}

\theoremstyle{nonumberplain}
\theoremheaderfont{\itshape}
\theorembodyfont{\normalfont}
\theoremseparator{.}
\theoremsymbol{\ensuremath{_\square}}
\newtheorem{proof}{Proof}
\qedsymbol{\ensuremath{_\square}}

\theoremstyle{empty}
\theoremsymbol{\ensuremath{_\square}}
\newtheorem{refproof}{Proof}
\qedsymbol{\ensuremath{_\square}}
\newcommand{\R}{\mathbb{R}}
\newcommand{\N}{\mathbb{N}}
\newcommand{\Z}{\mathbb{Z}}

\newcommand{\na}{(\{n\alpha\})_{n \geq 1}}
\newcommand{\ui}{[0,1)}
\newcommand{\ab}{[a,b)}
\newcommand{\xn}{(x_n)_{n \geq 1}}
\newcommand{\lc}{_{n\geq 1}}
\newcommand{\an}{(a_n)_{n \geq 1}}
\newcommand{\ana}{(\{a_n\alpha\})_{n \geq 1}}
\newcommand{\abm}{[\bm{a},\bm{b})}

\begin{document}
	\title{On Bounded Remainder Sets and \\ Strongly Non-Bounded Remainder Sets for Sequences $(\{a_n\alpha\})_{n\geq 1}$ }
	\author{Lisa Kaltenböck\footnote{The author is supported by the Austrian Science Fund (FWF), Project F5507-N26, which is a part of the Special Research Program Quasi-Monte Carlo Methods: Theory and Applications.}, Gerhard Larcher\footnote{The author is supported by the Austrian Science Fund (FWF), Project F5507-N26, which is a part of the Special Research Program Quasi-Monte Carlo Methods: Theory and Applications and Project I1751-N26.}}
	\date{\vspace{-5ex}}
	\maketitle
	
\begin{abstract}
	We give some results on the existence of bounded remainder sets (BRS) for sequences of the form $(\{a_n\alpha\})_{n\geq 1}$, where $(a_n)_{n\geq 1}$ - in most cases - is a given sequence of distinct integers. Further we introduce the concept of strongly non-bounded remainder sets (S-NBRS) and we show for a very general class of polynomial-type sequences that these sequences cannot have any S-NBRS, whereas for the sequence $(\{2^n\alpha\})_{n \geq 1}$ every interval is an S-NBRS.
\end{abstract}

\section{Introduction and Statement of Results}
Throughout this paper we will use $\{x\}$ to denote the fractional part of a number $x$ and $\|x\| = \min(\{x\},1-\{x\})$ to denote the difference between $x$ and the nearest integer.
\begin{definition}
	\label{def:ud}
	An infinite sequence $(x_n)_{n \geq 1}$ in $[0,1)$ is said to be uniformly distributed modulo 1 if for every interval $[a,b) \subseteq [0,1)$ we have	
	\begin{equation*}
	\lim\limits_{N \rightarrow \infty}\frac{\#\{1 \leq n \leq N: x_n \in [a,b)\}}{N} = b-a.
	\end{equation*}
\end{definition}
A most common measure for the quality of the uniform distribution of a sequence is the discrepancy $D_N$ of the sequence:

\begin{definition}
	For a sequence $(x_n)_{n \geq 1}$ in $[0,1)$ and any subinterval $[a,b) \subseteq [0,1)$, let 
	\begin{equation*}
	D_N([a,b)) = \frac{1}{N}\left(\#\{1 \leq n \leq N: x_n \in [a,b)\} - N(b-a)\right)
	\end{equation*}
	be the ($N$-th) discrepancy function of $[a,b)$.
	
	Then, the ($N$-th) discrepancy of $(x_n)$ is given by 
	\begin{equation*}
		D_N = \sup_{[a,b)} |D_N([a,b))|.
	\end{equation*}
\end{definition}

A famous result of W. M. Schmidt \cite{SchmidtVII} gives the following general lower bound for the discrepancy of an arbitrary sequence in $[0,1)$:
\begin{thmm}[Schmidt]
	There exists a constant $c > 0$ such that for any infinite sequence $\xn$ in $\ui$ it holds that
	\begin{equation*}
	D_N \geq c \frac{\log N}{N}
	\end{equation*}
	for infinitely many $N$.
\end{thmm}
This result, however, does not contain any lower bound for the absolute value $|D_N([a,b))|$ of the discrepancy function of a concrete fixed interval $[a,b)$. Such a - essentially best possible - bound was given by Tijdeman and Wagner in \cite{Tijdeman1980}.
\begin{thmm}[Tijdeman, Wagner]
	For every sequence $\xn$ we have for almost all $b \in \ui$ that
	\begin{equation*}
	|D_N([0,b))| \geq \frac{1}{400}\frac{\log N}{N}
	\end{equation*}
	for infinitely many $N$.
\end{thmm}
The other extremal case are sets $[a,b)$ of bounded remainder. 
\begin{definition}
	Let $(x_n)_{n\geq 1}$ be a sequence in $[0,1)$ and let $[a,b) \subseteq [0,1)$. $[a,b)$ is called bounded remainder set (BRS) if there exists some constant $c$ such that
	\begin{equation*}
	|D_N([a,b))| \leq c \frac{1}{N}
	\end{equation*}
	for all $N \in \N$.
\end{definition}
Note that for all $\ab$, $|D_N([a,b))| \geq \frac{1}{4N}$ for infinitely many $N$.
It was shown by Schmidt \cite{SchmidtVI} that for every sequence $(x_n)_{n \geq 1}$ in $[0,1)$ the set of intervals $[0,b)$ which are BRS is countable. 

For some of the most well-known and commonly studied sequences $(x_n)_{n \geq 1}$ in $[0,1)$ the bounded remainder sets are explicitly known. For example for the one-dimensional Kronecker sequence $(\{n\alpha\})_{n \geq 1}$, where $\alpha$ is a fixed irrational, by Kesten \cite{Kesten1966} the following was shown:
\begin{thmm}[Kesten]
	\label{thm:Kesten}
	An interval $\ab \subseteq \ui$ is a BRS for the sequence $\na$ if and only if
	\begin{equation*}
	b - a = \{j\alpha\}
	\end{equation*}
	for some $j \in \Z$.
\end{thmm}
Kesten has given an elementary proof. Proofs and generalizations of this result based on ergodic theory and topological dynamics were given for example by Furstenberg, Keynes and Shapiro \cite{Furstenberg1973}, Petersen \cite{Petersen1973} and Oren \cite{Oren1982}.

Theorem \ref{thm:Kesten} shows that the collection of BRS of $\na$ is "dense" in $\ui$, meaning that for all $\varepsilon > 0$ and all $x,y \in \ui$ there is a BRS $\ab$ with $|a - x| < \varepsilon$ and $|b - y| < \varepsilon$. We say: \textit{the sequence has a dense collection of BRS}.

A deep result based on the techniques from ergodic theory was given by Liardet \cite{Liardet1987} in the case of "polynomial Kronecker sequences" of higher degree:
\begin{thmm}[Liardet]
	\label{thm:Liardet}
	Let $p(x) \in \R[x]$ of degree $d \geq 2$ and with irrational leading coefficient. Then the only BRS of the sequence $(\{p(n)\})\lc$ are the empty set and $\ui$.	
\end{thmm}
So for example the sequence $(\{n^2\alpha\})_{n \geq 1}$ or more generally a sequence $(\{f(n)\alpha\})_{n\geq 1}$ with $f(x) \in \Z[x]$ of degree $d \geq 2$, does not have any non-trivial BRS. To our best knowledge until now there is no elementary proof of this fact.
A further nice result was given by Tijdeman and Voorhoeve in \cite{Tijdeman1981}, where they in some sense classify a collection $S$ of intervals in $\ui$ such that there exists a sequence $\xn$ having all elements of $S$ as BRS.
The notion of BRS was also studied in higher dimensions (in general and for concrete sequences) and for continuous motions for example in \cite{Hellekalek1984}, \cite{Grepstad2014}, \cite{Liardet1987}, \cite{Rauzy}, \cite{Ferenczi} or \cite{GrepstadLarcher}.

Our investigations of this paper started when we were asked whether the sequence $\xn = (\{2^n\alpha\})\lc$ has non-trivial BRS or not.
This question can be answered very easily even in the more general case of Pisot numbers $\beta$ as bases.
Here, we restrict ourselves to $\alpha$ whose $\beta$-expansion is normal to base $\beta$ and thus for which can be guaranteed that $(\{\beta^n \alpha\})_{n \geq 1}$ is uniformly distributed, as was shown by Bertrand-Mathis \cite{Bertrand}. For the special case that $\beta > 1$ is an integer, the normality of $\beta$ and the uniform distribution of $(\{\beta^n \alpha\})_{n \geq 1}$ are equivalent.
\begin{thm}
	\label{thm:brs_normal_alpha}
	Let $\beta > 1$ be a Pisot number and let $\alpha$ be normal to base $\beta$. Then the sequence $(\{\beta^n \alpha\})_{n \geq 1}$ does not have non-trivial bounded remainder sets.
\end{thm}
As will be seen from the proof of this theorem in Section \ref{sec:proofs_1}, for these sequences all intervals $[a,b) \subset \ui$ with $a \neq 0$ and $b \neq 1$ (if $\beta$ is an integer this restriction can be omitted) even are non-bounded remainder sets in a very strong sense.
Namely, all of these intervals are \textit{strongly non-bounded remainder sets}:
\begin{definition}
	Let $\xn$ be a sequence in $\ui$. An interval $\ab \subset \ui$ is said to be a \textit{strongly non-bounded remainder set (S-NBRS)} if for all $K \in \N$ there are $K$ successive elements of $\xn$ which all are contained in $\ab$ or which all are not contained in $\ab$.
\end{definition}
Of course the concepts of BRS and S-NBRS can be defined quite analogously in higher dimensions.
It is clear that an S-NBRS is not a BRS, Lemma \ref{lma:consecutive_elements} relates the two notions. Earlier we have asked for an elementary proof of the result of Liardet (Theorem \ref{thm:Liardet}), that certain polynomial sequences do not have (non-trivial) BRS. A first idea to approach such a proof would be to try to proceed like in the proof of Theorem \ref{thm:brs_normal_alpha}, i.e., to try to show, that for such sequences any interval is an S-NBRS, like in the case of the sequence $(\{\beta^n\alpha\})\lc$. In Theorem \ref{thm:sn-brs_polynomial} we show in a very general setting that such an approach has to fail.
\begin{thm}
	\label{thm:sn-brs_polynomial}
	Let $\bm{p}(n) = (p^{(1)}(n), p^{(2)}(n), \dots, p^{(s)}(n))$, where all $p^{(i)}(n)$, $i \leq s$ are real polynomials with the property, that for each point $\bm{h} \in \Z^s, \bm{h} \neq 0$, the polynomial $\langle \bm{h}, \bm{p}(n) \rangle$ has at least one non-constant term with irrational coefficient. Then the sequence $(\{\bm{p}(n)\})\lc$ does not have any s-dimensional S-NBRS.
\end{thm}

Let us now consider BRS for sequences of the form $(\{a_n\alpha\})\lc$ where $\an$ is a given strictly increasing sequence of integers, and where $\alpha$ again is an irrational number. Distribution properties of such sequences from different points of view (uniform distribution, discrepancy, law of iterated logarithm, Poissonian pair correlation) were studied in a multitude of papers (see for example for a survey \cite{Drmota:Sequences_Discrepancies_Applications}, or for more recent references \cite{aistleitner_larcher_2016} or \cite{Aistleitner2017}).

We already know: if $a_n = n$, then we have a dense collection of BRS, if $a_n = f(n)$ with $f(x) \in \Z[x]$ of degree $d \geq 2$, or if $a_n = \beta^n$ for some Pisot number $\beta$, then we do not have any BRS. 
In a first impression one could think that the existence of BRS for sequences $(\{a_n\alpha\})\lc$ relies on a "not too strong" growth rate of the sequence $\an$. For example one could have the idea, that sequences $(\{a_n\alpha\})\lc$ with $\an$ lacunary never can have BRS. This indeed, of course, is wrong. A very simple counterexample is the sequence $(\{a_n\alpha\})\lc$ with $a_n = n + q_n$ with $q_n$ a best approximation denominator of $\alpha$. Since $\|q_n\|$, i.e. the distance of $q_n$ to the nearest integer, always is very small, the sequence $(\{a_n\alpha\})\lc$ behaves quite similar to the sequence $\na$ and from this observation we can deduce that $(\{a_n\alpha\})\lc$ has a dense collection of bounded remainder sets. However, in general, $(\{a_n\alpha\})\lc$ and $\na$ do not have the same collection of BRS.
\begin{thm}
	\label{thm:lacunary_brs}
	\begin{enumerate}
		\item Let $\alpha$ be irrational and let $q_n$ be the denominator of the $n$-th convergent from the continued fraction expansion of $\alpha$. Then every interval of the form $[\{k_1\alpha\},\{k_2\alpha\})$ with $k_1,k_2 \in \N$ and $\{k_1\alpha\} < \{k_2\alpha\}$ is a BRS for the uniformly distributed sequence $(\{(q_n+n)\alpha\})\lc$. 
		\item In general, the sequence $(\{(q_n+n)\alpha\})\lc$ does not have the same collection of BRS as the sequence $\na$.
	\end{enumerate}
\end{thm}
Of course BRS for a sequence $(\{a_n\alpha\})\lc$ for an (ultimately) increasing $\an$ only can exist if $a_n$ grows at least linearly. This is a simple fact but for the sake of completeness we formulate it as a theorem:
\begin{thm}
	\label{thm:slow_growth_rate}
	Let $\an$ be an increasing sequence of integers and let $\ana$ be uniformly distributed. If  
	\begin{equation*}
	\lim_{n\rightarrow \infty}\frac{a_n}{n} = 0,
	\end{equation*}
	then every interval $\ab \subset \ui$ of positive measure is an S-NBRS.
\end{thm}
For every other growth rate, however, it is possible to achieve even a dense collection of BRS.
\begin{thm}
	\label{thm:arbitrary_growth_rate_Brs}
	Let $\varphi(n)$ be any strictly increasing sequence with $\varphi(n) \geq 2n$ for all $n \in \N$. Then for all irrational $\alpha$ with bounded continued fraction coefficients there is a constant $L$ and a strictly increasing sequence of integers $\an$, with $\varphi(n) \leq a_n \leq L\varphi(n)$ for all $n \geq 1$ such that $\ana$ is uniformly distributed and such that $\ana$ has a dense collection of BRS.
\end{thm}
In Theorem \ref{thm:lacunary_brs} and in Theorem \ref{thm:arbitrary_growth_rate_Brs} we have given examples of sequences $\ana$ with dense collections of BRS for always one choice of $\alpha$. But it is also possible for any countable set of irrationals $\alpha$ to give an integer sequence $\an$ such that $\ana$ has a dense collection of BRS for all of these $\alpha$.
\begin{thm}
	\label{thm:countably_many_alpha}
	Let $\mathcal{A}$ be a countable set of irrational numbers. Then there is a lacunary sequence $\an$ of integers such that $\ana$ is uniformly distributed and such that $\ana$ has a dense collection of BRS for all $\alpha \in \mathcal{A}$.
\end{thm}
In Section \ref{sec:proofs_1} we will give the proofs of the Theorems \ref{thm:brs_normal_alpha}, \ref{thm:lacunary_brs}, \ref{thm:slow_growth_rate}, \ref{thm:arbitrary_growth_rate_Brs}, and \ref{thm:countably_many_alpha} and Section \ref{sec:proof_2} is dedicated to the proof of Theorem \ref{thm:sn-brs_polynomial}.

%
%
%
%
%
%

\section{Proofs of Theorems \ref{thm:brs_normal_alpha}, \ref{thm:lacunary_brs}, \ref{thm:slow_growth_rate}, \ref{thm:arbitrary_growth_rate_Brs}, and \ref{thm:countably_many_alpha}}
\label{sec:proofs_1}

The first lemma characterizes the relation between bounded remainder sets and \linebreak
S-NBRS and will be used to prove Theorem \ref{thm:brs_normal_alpha}.

\begin{lemma}
	\label{lma:consecutive_elements}
	Let $(x_n)\lc$ be a sequence in $[0,1)$. 
	\begin{enumerate}
		\item If the interval $[a,b) \subset [0,1)$ is an S-NBRS, then $[a,b)$ is not a BRS.
		\item If $(x_n)\lc$ has a dense collection of BRS, then there do not exist any S-NBRS.
	\end{enumerate}
\end{lemma}

\begin{proof}
	For the proof of the first part consider intervals $[a,b) \subset [0,1)$ which are bounded remainder sets for an arbitrary sequence $(x_n)\lc$. By definition, there exists a constant $c$ independent of $N$ such that
	\begin{equation*}
	|\#\{1 \leq n \leq N: x_n \in [a,b)\} - (b-a) N| \leq c
	\end{equation*}
	for all $N \in \N$.
	Let $k \in \N, k > \frac{2c}{1-b+a}$ and first suppose that there exists an $N$ such that 
	\begin{equation*}
	\#\{N < n \leq N+k: x_n \in [a,b)\} = k.
	\end{equation*}
	Considering this particular $N$ we have that
	\begin{align*}
	\#\{1 \leq n \leq N+k&: x_n \in [a,b)\} - (b-a)(N+k) \\
	&= \#\{1 \leq n \leq N: x_n \in [a,b)\}\\
	&\quad + \#\{N + 1 \leq n \leq N+k: x_n \in [a,b)\} 
	- (b-a)(N+k) \\
	&\geq -c + k -(b-a)k \\
	&> c,
	\end{align*}
	which is a contradiction.
	On the other hand, if some $k > \frac{2c}{b-a}$ successive elements are not contained in $\ab$ we have
	\begin{align*}
	\#\{1 \leq n \leq N+k&: x_n \in [a,b)\} - (b-a)(N+k) \\
	&= \#\{1 \leq n \leq N: x_n \in [a,b)\}	- (b-a)(N+k) \\
	&\leq c -(b-a)k < -c,
	\end{align*}
	Therefore, an S-NBRS cannot be a BRS.
	
	For the second part of the lemma let $(x_n)\lc$ be a sequence with a dense collection of bounded remainder sets and consider $[a,b) \subset [0,1)$.
	Then there exists an interval $[\underline{a},\underline{b}) \subseteq [a,b)$ which is a BRS for $(x_n)\lc$. Since there exists $K$ such that at most $K$ consecutive elements of $(x_n)\lc$ are not contained in $[\underline{a},\underline{b})$ this also holds for the interval $\ab$.
	Moreover, if $a \neq 0$ and $b \neq 1$ there also exists $[\bar{a},\bar{b}) \supset \ab$, which is a BRS and thus the number of consecutive elements of $\xn$ in $[\bar{a},\bar{b})$ as well as in $\ab$ is bounded.
	In case that $a = 0$ or $b = 1$ we consider the complement of $\ab$. Since the number of successive elements of $\xn$ which are not contained in $\ui \setminus \ab$ is bounded, so is the number of successive elements which are contained in $\ab$.
	Therefore, no interval $[a,b)$ can be an S-NBRS.
\end{proof}

\begin{refproof}[Proof of Theorem \ref{thm:brs_normal_alpha}.]
	We exclude the unit interval $[0,1)$ from our considerations because it trivially fulfills the properties of a bounded remainder set.
	
	Let $\beta > 1$ be a Pisot number and let $\alpha \in \ui$ with expansion 
	\begin{equation*}
	\alpha = \sum_{i = 1}^{\infty} \alpha_i\beta^{-i},
	\end{equation*}
	where $\alpha \in \{0,1, \dots, \lceil\beta\rceil -1\}$ and such that $\sum_{i > j} \alpha_i\beta^{-i} < \beta^{-j}$ for all $j \geq 0$. Moreover, we choose $\alpha$ to be normal to base $\beta$.
	At first, consider intervals of the form $\ab$, where $a \neq 0$ and $b \neq 1$ and define $\varepsilon := \min\left(\frac{\|a\|}{2}, \frac{\|b\|}{2}\right)$. 
	
	Since $\beta$ is a Pisot number, $\beta^n$ is a good approximation of an integer, i.e. $\|\beta^n\|$ converges to zero at an exponential rate. Therefore, $\sum_{n = 0}^{\infty}\|\beta^n\| < c$ for some constant $c$ and by $m$ we denote the first index such that
	\begin{equation*}
	\sum_{n = m}^{\infty} \|\beta^n\| < \frac{1}{\lceil\beta\rceil -1} \varepsilon.
	\end{equation*}
	
	Furthermore, consider the $\beta$-expansion of $\varepsilon$, i.e.
	\begin{equation*}
	\varepsilon = \sum_{i = 1}^{\infty}\varepsilon_i\beta^{-i},
	\end{equation*}
	with coefficients $\varepsilon_i \in \{0,1, \dots,\lceil\beta\rceil -1\}$. Let $j$ denote the index of the first coefficient for which $\varepsilon_j \neq 0$.
	A sufficient condition for  $\sum_{i = n+1}^{\infty}\alpha_i\beta^{n-i} < \varepsilon$ is then given by $\alpha_{n+1} = \dots = \alpha_{n+j} = 0$.
	
	From the fact that $\alpha$ is normal to base $\beta$, i.e., every admissible finite block $d$ of integers $d_i \in \{0,1, \dots, \lceil\beta\rceil -1\}$ occurs with the right frequency in the $\beta$-expansion of $\alpha$, the occurrence of arbitrary long blocks consisting only of zeros in the base $\beta$ expansion of $\alpha$ follows. In other words, for all $k \in \N$ there exists an index $N$ for which it holds that
	
	\begin{equation*}
	\alpha_{N+1} = \alpha_{N+2} = \dots = \alpha_{N + m + k + j} = 0.
	\end{equation*}
	For the elements $\{\beta^n\alpha\}$ with $N + m < n \leq N + m + k$ it holds that $\alpha_{n-i} = 0$ for $i = 0, 1, \dots, m - 1$ which implies that
	\begin{equation*}
	\|\sum_{i = 0}^{n-1}\alpha_{n-i}\beta^{i}\| = \|\sum_{i = m}^{n-1}\alpha_{n-i}\beta^{i}\| 
	\leq \sum_{i = m}^{n-1}\alpha_{n-i}\|\beta^{i}\| < \varepsilon.
	\end{equation*}
	Additionally, we have $\alpha_i = 0$ for $i = n +1, \dots, n + j$ and $N + m < n \leq N + m + k$. Hence, 
	\begin{equation*}
	\sum_{i = n+1}^{\infty}\alpha_i\beta^{n-i} < \varepsilon.
	\end{equation*}
	Since $\varepsilon < 1/4$ it holds that
	\begin{align*}
	\|\beta^n\alpha\| 
	&= \|\sum_{i = 0}^{n-1}\alpha_{n-i}\beta^{i} + \sum_{i = n+1}^{\infty}\alpha_i\beta^{n-i}\| \\
	&\leq \|\sum_{i = 0}^{n-1}\alpha_{n-i}\beta^{i}\| + \sum_{i = n+1}^{\infty}\alpha_i\beta^{n-i} \\
	&< 2\varepsilon = \min(\|a\|,\|b\|)
	\end{align*}
	for $N + m < n \leq N + m + k$.
	Therefore, for any $k \in \N$ there exist $k$ consecutive elements $\{\beta^n\alpha\}$ which are not contained in the interval $\ab$, which therefore is an S-NBRS.
	
	Note that if $\beta$ is an integer it holds that $\{\sum_{i = 0}^{n-1}\alpha_{n-i}\beta^{i}\} = 0$ and
	\begin{equation*}
	\{\beta^n\alpha\} = \sum_{i = n+1}^{\infty}\alpha_i\beta^{n-i}.
	\end{equation*} 
	Thus, the occurrence of arbitrary long blocks of zeros implies the existence of arbitrarily many successive elements which are contained in some interval $[0,a)$. It follows that in this case even intervals of the form $[0,a)$ or $[b,1)$ are S-NBRS.
	
	Nevertheless, for arbitrary base $\beta$ it remains to show that intervals $[0,a)$ or $[b,1)$ cannot be BRS as well. From the previous considerations we know that the number of consecutive elements of $(\{\beta^n\alpha\})\lc$ which are not contained in some interval $\ab$ with $0 < a < b <1$ is unbounded, i.e. for any  $k$ there exists an $N$ such that $\{\beta^n\alpha\} \in [0,a)\cup[b,1)$ for all $N < n \leq N + k$. Now consider a strictly increasing sequence of such integers $(k_i)$ with corresponding indices $N_i$ and let $\varepsilon = \frac{b-a}{3}$.
	
	Then, let
	\begin{equation*}
	A_ik_i = \#\{N_i < n \leq N_i + k_i : \{\beta^n\alpha\} \in [0,a)\}
	\end{equation*}
	and
	\begin{equation*}
	B_ik_i = \#\{N_i < n \leq N_i + k_i : \{\beta^n\alpha\} \in [b,1)\}.
	\end{equation*}
	Since $A_i + B_i = 1$ and $1 - (b-a) + 2\varepsilon < 1$ it either holds that
	\begin{equation*}
	A_i > a + \varepsilon \qquad \text{for infinitely many i},
	\end{equation*}
	or 
	\begin{equation*}
	B_i > 1-b + \varepsilon \qquad \text{for infinitely many i}.
	\end{equation*}
	W.l.o.g. let us assume the first case and suppose that there exists a constant $c$ such that
	\begin{equation*}
	ND_N([0,a)) = |\#\{1 \leq n \leq N: \{\beta^n\alpha\} \in [0,a)\} - N a| \leq c
	\end{equation*}
	for all $N \in \N$. We use this property for $N_i$ and obtain
	\begin{align*}
	(N_i + k_i)D_{N_i + k_i}([0,a)) &= \#\{1 \leq n \leq N_i + k_i : \{\beta^n\alpha\} \in [0,a)\} - (N_i+k_i)a\\
	&\geq A_i k_i - a k_i -c \\
	&> \varepsilon k_i -c.
	\end{align*}
	Since $\varepsilon k_i$ is unbounded, so is $(N_i + k_i)D_{N_i + k_i}([0,a))$, and $[0,a)$ as well as $[a,1)$ therefore cannot be BRS.
\end{refproof}

The next lemma will be of use in the proofs of Theorem \ref{thm:lacunary_brs}, \ref{thm:arbitrary_growth_rate_Brs} and \ref{thm:countably_many_alpha}.
\begin{lemma}
	\label{lma:dcbrs_ud}
	Let $(x_n)_{n \geq 1}$ be a sequence in $[0,1)$ which has a dense collection of bounded remainder sets. Then $(x_n)_{n \geq 1}$ is uniformly distributed modulo 1.  
\end{lemma}

\begin{proof}
	Let $a,b \in (0,1)$. Since $(x_n)_{n \geq 1}$ has a dense collection of bounded remainder sets, for every $\varepsilon > 0$ there exist $\underline{a}, \bar{a}, \underline{b}, \bar{b}\in (0,1)$ such that 
	\begin{align*}
	\underline{a} \leq a \leq \bar{a}, \quad  a - \underline{a} &\leq \varepsilon, \quad \bar{a} - a \leq \varepsilon, \\
	\underline{b} \leq b \leq \bar{b}, \quad  b - \underline{b} &\leq \varepsilon, \quad \bar{b} - b \leq \varepsilon,
	\end{align*}
	and $[\underline{a}, \bar{b})$, $[\bar{a},\underline{b})$ are bounded remainder sets.
	From
	\begin{align*}\\
	\#\{1 \leq n \leq N: x_n \in [\underline{a},\bar{b})\} 
	&\geq \#\{1 \leq n \leq N: x_n \in [a,b)\} \\
	&\geq \#\{1 \leq n \leq N: x_n \in [\bar{a},\underline{b})\} 
	\end{align*}
	for all $N \in \N$, it follows that
	\begin{align*}
	\lim\limits_{N \rightarrow \infty} \frac{\#\{1 \leq n \leq N: x_n \in [\underline{a},\bar{b})\}}{N}	
	&\geq \lim\limits_{N \rightarrow \infty} \frac{\#\{1 \leq n \leq N: x_n \in [a,b)\}}{N} \\
	&\geq \lim\limits_{N \rightarrow \infty} \frac{\#\{1 \leq n \leq N: x_n \in [\bar{a},\underline{b})\}}{N}.
	\end{align*}
	By the definition of bounded remainder sets it holds that 
	\begin{equation*}
	\lim\limits_{N \rightarrow \infty} \frac{\#\{1 \leq n \leq N: x_n \in [\underline{a},\bar{b})\}}{N} = \bar{b} - \underline{a} \leq b - a + 2\varepsilon,
	\end{equation*}
	and analogously
	\begin{equation*}
	\lim\limits_{N \rightarrow \infty} \frac{\#\{1 \leq n \leq N: x_n \in [\bar{a},\underline{b})\}}{N} = \underline{b} - \bar{a} \geq b - a - 2\varepsilon.
	\end{equation*}
	Since $\varepsilon$ was arbitrary small, we have that
	\begin{equation*}
	\lim\limits_{N \rightarrow \infty} \frac{\#\{1 \leq n \leq N: x_n \in [a,b)\}}{N} = b - a,
	\end{equation*}
	and therefore $(x_n)_{n \geq 1}$ is uniformly distributed.
\end{proof}

\begin{refproof}[Proof of Theorem \ref{thm:lacunary_brs}.]
	For the proof of the first part, our goal is to show that there exists an index $k$ such that
	\begin{equation}
	\label{eq:equality_of_numbers}
	\#\{k \leq n \leq N: \{a_n\alpha\} \in J\} = \#\{k \leq n \leq N: \{n\alpha\} \in J\}
	\end{equation}
	for all $N \in \N$ where $J = [\{k_1\alpha\},\{k_2\alpha\})$ with $k_1,k_2 \in \N$ and $\{k_1\alpha\} < \{k_2\alpha\}$ and where $a_n = n + q_n$ with $q_n$ the denominator from the $n$-th convergent of the continued fraction expansion of $\alpha$.
	Since $J$ is a BRS for $(\{n\alpha\})_{n \geq 1}$ this implies that
	\begin{equation*}
	|\#\{1 \leq n \leq N: \{a_n\alpha\} \in J\} - \lambda(J)N| \leq c + k
	\end{equation*}
	for all $N \in \N$ with the right side of the inequality independent of $N$ and where $\lambda(J) = \{k_2\alpha\} - \{k_1\alpha\}$.
	
	Equation (\ref{eq:equality_of_numbers}) is fulfilled if we are able to show that the point $\{a_n\alpha\}, n \geq k$ is a suitable approximation of $\{n\alpha\}$. More precisely, if the distance between $\{a_n\alpha\}$ and $\{n\alpha\}$ is smaller than the distance between the boundaries of $J$ and $\{n\alpha\}$.
	
	Let us begin with the distance between $\{a_n\alpha\}$ and $\{n\alpha\}$. This distance is given by
	\begin{align*}
	\| a_n\alpha - n\alpha \| &= \| (q_n + n)\alpha - n\alpha \| \\
	&= \| q_n\alpha \|.
	\end{align*}
	On the other hand, the distance between $\{n\alpha\}$ and the boundaries of $J$ can be estimated by the minimal distance between any two elements $0, \{\alpha\}, \{2\alpha\},...,\{n\alpha\}$ if $\{k_1\alpha\}$ and $\{k_2\alpha\}$ are contained in those points, i.e. if $n > \max(k_1, k_2)$. For this minimal distance it holds that
	\begin{align*}
	\min_{0 \leq l_1 < l_2 \leq n} \|l_2\alpha - l_1\alpha \| 
	&= \min_{0 \leq l_1 < l_2 \leq n} \| \underbrace{(l_2-l_1)}_{\in \{1, \dots, n\}}\alpha  \| \\
	&= \min_{1 \leq l \leq n}\|l\alpha\|.
	\end{align*}
	
	Since the convergents of the continued fraction expansion of $\alpha$ provide best rational approximations, it holds that
	\begin{equation*}
	\|q_n\alpha\| < \min_{1 \leq l \leq n}\|l\alpha\|,
	\end{equation*}
	if $q_n > n$ which is certainly the case for $n \geq 4$.
	
	Summing up, we have to choose the index $k$ as $k = \max(k_1+1,k_2+1,4)$ such that the boundaries of $J$ are contained in the considered points and to guarantee the approximation property. For this choice of $k$, equation (\ref{eq:equality_of_numbers}) holds and $J$ therefore is a bounded remainder set for the sequence $(\{a_n\alpha\})_{n\geq 0})$. The uniform distribution of this sequence directly follows with Lemma \ref{lma:dcbrs_ud} since the intervals $J$ form a dense collection of BRS.
	
	\ \\
	
	For the proof of the second part, we consider $\alpha = [0;1,1,\bar{2}]$ and show that some interval $J = [a, a+ \alpha)$, which is a BRS for $\na$ is not a BRS for the sequence $(\{(q_n + n)\alpha\})_{n\geq 1})$. More precisely, we show the existence of an $a \in \R$, obviously $a \neq \{j\alpha\}$ for any $j \in \Z$, for which
	\begin{equation}
	\label{eq:infinit}
	\{n\alpha\} \notin J \text{ but } \{(q_n + n)\alpha\} \in J \qquad \text{for infinitely many } n \in \N
	\end{equation}
	and
	\begin{equation}
	\label{eq:finite}
	\{n\alpha\} \in J \text{ but } \{(q_n + n)\alpha\} \notin J \qquad \text{for only finitely many } n \in \N.
	\end{equation}
	This would imply that $|\#\{1 \leq n \leq N: \{(q_n + n)\alpha\} \in J\} - \alpha N|$ is unbounded and that $J$ is not a BRS for $(\{(q_n + n)\alpha\})_{n\geq 0})$ whereas it is a BRS for $\na$.
	
	Note that because of the special form of $\alpha$, the denominator $q_n$ of the $n$-th convergent of $\alpha$, is even if $n$ is even and odd if $n$ is odd.
	Moreover, it is a well known fact that $p_n/q_n$ is smaller than $\alpha$ for every
	even value of $n$ and greater for every odd value of $n$, which implies that $0 < \{q_n\alpha\} < 1/2$ for $n$ even and $1/2 < \{q_n\alpha\} < 1$ for $n$ odd. From the proof of the first part we know that the distance between $\{n\alpha\}$ and $\{(q_n+n)\alpha\}$ is given by $\|q_n\alpha\|$.
	Hence, 
	\begin{align*}
	\{n\alpha\} &< \{(q_n+n)\alpha\} \quad \text{if } n \text{ is even},\\
	\{n\alpha\} &> \{(q_n+n)\alpha\} \quad \text{if } n \text{ is odd}.
	\end{align*}
	
	To begin with, we want to construct $a$ such that there exists a subsequence $(\{n_i\alpha\})_{i \geq 1}$ for which it holds that 
	\begin{equation}
	\label{eq:a}
	\{n_i\alpha\} < a \text{ and } a \leq \{(q_{n_i}+n_i)\alpha\} < a + \alpha.
	\end{equation}
	This would guarantee property (\ref{eq:infinit}).
	Therefore, choose any even $n_1 \in \N$ such that $\{(q_{n_1}+n_1)\alpha\} + \alpha < 1$ and with $\{q_{n_1}\alpha\} < \alpha$. The first condition is necessary because from (\ref{eq:a}) we know that $a \in (\{n_1\alpha\}, \{(q_{n_1}+n_1)\alpha\}]$ and it therefore guarantees that $a < a + \alpha$. The condition on $\{q_{n_1}\alpha\}$ ensures that $\{(q_{n_i}+n_i)\alpha\} < a + \alpha$.
	
	Now, we want to choose $n_2$ such that $\{n_2\alpha\}$ lies between $\{n_1\alpha\}$ and $\{(q_{n_1}+n_1)\alpha\}$. It follows from the approximation property of $q_n$ that $n_2 = n_1 + q_{n_1+2}$ is a suitable choice. In fact, of the first $n_2$ elements of $\na$, $\{n_2\alpha\}$ is the one closest to $\{n_1\alpha\}$. Moreover, $n_2$ again is an even number which implies that 
	\begin{equation*}
	\{n_1\alpha\} < \{n_2\alpha\} < \{(q_{n_2}+n_2)\alpha\} < \{(q_{n_1}+n_1)\alpha\}.
	\end{equation*}
	Therefore, by (\ref{eq:a}) it has to hold that $a \in (\{n_2\alpha\}, \{(q_{n_2}+n_2)\alpha\}]$. In the same manner, $n_2$ now can be used to construct $n_3$ and in general we obtain
	\begin{equation*}
	n_{i+1} = n_i + q_{n_i + 2}.
	\end{equation*}
	This subsequence satisfies
	\begin{equation*}
	\{n_i\alpha\} < \{n_{i+1}\alpha\} < \{(q_{n_{i+1}}+n_{i+1})\alpha\} < \{(q_{n_i}+n_i)\alpha\}
	\end{equation*}
	for all $i \in \N$.
	Hence, equation (\ref{eq:a}) is fulfilled if we choose
	\begin{equation*}
	a = \lim\limits_{n \rightarrow \infty} \{(q_{n_i}+n_i)\alpha\}
	\end{equation*}
	and it holds that 
	\begin{equation*}
	\{n_i\alpha\} \notin [a,a+\alpha) \quad \text{ but } \quad \{(q_{n_i} + n_i)\alpha\} \in [a,a+\alpha).
	\end{equation*}
	\begin{figure}[h]
		\centering
		\includegraphics[width=1\textwidth]{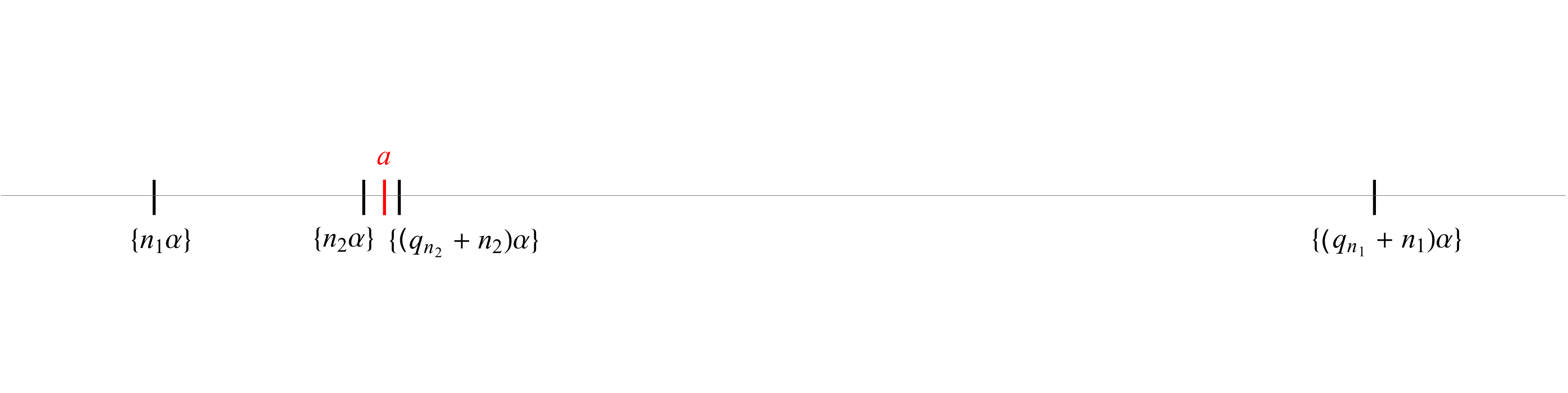}
		\caption{Illustration of $n_1$ and $n_2$}\label{fig:counterexample}
	\end{figure}
	
	It remains to check property (\ref{eq:finite}).
	First, we want to show that there is no point $\{n\alpha\}$ with $n > n_1$ for which it holds that 
	\begin{equation*}
	\{n\alpha\} > a \quad \text{ and } \quad \{(q_{n}+n)\alpha\} < a.
	\end{equation*}
	This can only be the case if $n$ is odd.
	Let $i \in \N$ be the index such that $n_i < n < n_{i+1}$.
	It holds that $\{n\alpha\} > \{n_{i+1}\alpha\}$. Additionally, $\{n\alpha\} > \{(q_{n_i}+n_i)\alpha\}$ and $\{(q_n + n)\alpha\} > \{(q_{n_i}+n_i)\alpha\}$ because otherwise
	$\{n\alpha\}$ and $\{(q_n + n)\alpha\}$ would be a better approximation of $\{n_{i+1}\alpha\}$ than $\{(q_{n_i}+n_i)\alpha\}$ which cannot be the case since $q_{n_i}+n_i > q_n + n$ and $q_{n_i}$ is a best approximation denominator.
	Therefore, for all $n \in \N$, $n > n_1$ it holds that
	\begin{equation}
	\label{eq:first}
	\{n\alpha\} > a \implies \{(q_{n}+n)\alpha\} > a.
	\end{equation}
	
	Secondly, we also have to check the right boundary of the interval $J$, that means we want to show that there are only finitely many $n$ with the property
	\begin{equation}
	\label{eq:a+alpha}
	\{n\alpha\} < a + \alpha \quad \text{ and } \quad \{(q_{n}+n)\alpha\} \geq a + \alpha.
	\end{equation}
	Note that if (\ref{eq:a+alpha}) holds then $n$ has to be even and that the distance between $\{n\alpha\}$ and $\{(q_{n}+n)\alpha\}$ is given by 	$\{q_n\alpha\}$. Therefore, the distance between $\{n\alpha\}$ and $a + \alpha$ is at most $\{q_n\alpha\}$. If $a > \{q_n\alpha\}$, which is certainly fulfilled for $n$ large enough (e.g. if $q_n > n_1$), we also obtain information about the distance of $\{(n-1)\alpha\}$ to the left boundary of $J$, namely
	\begin{equation}
	\label{eq:left_boundary}
	0 < a - \{(n-1)\alpha\} < \{q_n\alpha\}.
	\end{equation}
	For those $n$, we again have to consider two cases, depending on the position of \linebreak
	$\{(n-1)\alpha\}$.
	The first one is the case $\{(n-1)\alpha\} < \{n_1\alpha\}$. This however implies that 
	\begin{equation*}
	a - \{(n-1)\alpha\} > a - \{n_1\alpha\}.
	\end{equation*}
	Since $a - \{n_1\alpha\} > \{n_2\alpha\} - \{n_1\alpha\} = \{q_{n_1 + 2}\alpha\}$ (this can be easily seen from Figure \ref{fig:counterexample}) we have that 
	\begin{equation*}
	\{q_n\alpha\} > \{q_{n_1 + 2}\alpha\}.
	\end{equation*}
	Both, $n$ and $n_1 + 2$ are even numbers, therefore this implies that $n \leq n_1$.
	
	For the second case, $\{(n-1)\alpha\} > \{n_1\alpha\}$, observe that the points $\{n_i\alpha\}$ can get arbitrarily close to $a$. Therefore, for all $n > n_1$ which fulfill (\ref{eq:a+alpha}) there is an $i \in \N$ such that $\{n_i\alpha\} < \{(n-1)\alpha\} < \{n_{i+1}\alpha\}$.
	As in the previous case, from the right side of (\ref{eq:left_boundary}) and because $a - \{n_{i+1}\alpha\} > \{n_{i+2}\alpha\} - \{n_{i+1}\alpha\} = \{q_{n_{i+1}+2}\alpha\}$ it follows that
	\begin{equation*}
	\{q_n\alpha\} > \{q_{n_{i+1}+2}\alpha\},
	\end{equation*}
	which implies that $n \leq n_{i+1}$.
	Since of the first $n_{i+1}$ elements of $\na$, $\{n_{i+1}\alpha\}$ is the closest one to $\{n_{i}\alpha\}$, additionally it has to hold $n > n_{i+1}$. Thus, such an $n$ cannot exist.
	
	Summing up, for all $n \in \N$, $n > n_1$ it holds that
	\begin{equation}
	\label{eq:second}
	\{n\alpha\} < a + \alpha \implies \{(q_n + n)\alpha\} < a + \alpha
	\end{equation}
	By putting (\ref{eq:first}) and (\ref{eq:second}) together we obtain that
	\begin{equation*}
	\{n\alpha\} \in [a, a+ \alpha) \implies \{(q_n +n)\alpha\} \in [a, a+ \alpha)
	\end{equation*}
	for all $n > n_1$. Thus, property (\ref{eq:finite}) holds and the proof is complete.
\end{refproof}

\begin{refproof}[Proof of Theorem \ref{thm:slow_growth_rate}.]
	Let $\ab \subset \ui$.	
	The condition that the increasing integer sequence $(a_n)\lc$ grows slower than linearly implies that for all $k \in \N$ there exists $N_k \in \N$ such that $a_{N_k+1} = a_{N_k + 2} = \dots = a_{N_k + k}$.
	Thus, for each $k$ the elements $\{a_n\alpha\}$, with $N_k < n \leq N_k + k$ are either all contained or not contained in $\ab$ which therefore is an S-NBRS.
\end{refproof}

\begin{refproof}[Proof of Theorem \ref{thm:arbitrary_growth_rate_Brs}.]
	Let $\alpha$ have bounded continued fraction coefficients, i.e. 
	\linebreak $\alpha = [0;\alpha_1,\alpha_2,...]$ with $\alpha_i \leq L-1$ for all $i \in \N$ and let $q_n$ denote the $n$-th best approximation denominator of $\alpha$.
	Consider the sequence $a_n = q_{n'} + n$ where for each $n$ the parameter $n'$ is chosen such that
	\begin{equation*}
	q_{n' - 1} + n < \varphi(n) \leq q_{n'} + n.
	\end{equation*}
	It holds that
	\begin{align*}
	\varphi (n) \leq  a_n = q_{n'} + n 
	&= \alpha_{n'}q_{n' -1} + q_{n' - 2} + n \\
	&\leq (L-1)q_{n' -1} + q_{n' - 2} + n \\
	&\leq Lq_{n' -1} + n\\
	&< L (\varphi(n) - n) + n \\
	&< L \varphi(n).
	\end{align*}
	Moreover, we already know from the proof of the first part of Theorem \ref{thm:lacunary_brs} that $(\{a_n\alpha\})_{n \geq 1}$ has a dense collection of BRS if the distance between $\{a_n\alpha\}$ and $\{n\alpha\}$, which is given by $\|\{q_{n'}\alpha\}\|$, is smaller than $\min_{1 \leq l \leq n} \|\{l\alpha\}\|$ for all $n \in \N$.
	Therefore, from the approximation property of $q_{n'}$ and from the fact that
	\begin{equation*}
	q_{n'} \geq \varphi(n) - n > n
	\end{equation*}
	it follows that $(\{a_n\alpha\})_{n \geq 1}$ has a dense collection of BRS. Thus, by Lemma \ref{lma:dcbrs_ud} $\ana$ is also uniformly distributed.
\end{refproof}

The proof of Theorem \ref{thm:countably_many_alpha} basically follows the same idea as the proof of the first part of Theorem \ref{thm:lacunary_brs}, $a_n$ is again of the form $q_n + n$. However, with the difference that $q_n$ now does not any more denote the denominator of the $n$-th convergent from the continued fraction expansion of some single $\alpha$. 
Instead, we will use the simultaneous version of Dirichlet's Approximation Theorem (see e.g. \cite{Hardy1975}): \newline
For all irrational $\alpha_1, ...., \alpha_k$ and all $N \in \N$ there exist integers $p_1,...,p_k$ and $1 \leq q \leq N$ such that
\begin{equation*}
\left|\alpha_i - \frac{p_i}{q}\right| \leq \frac{1}{q N^{1/k}}.
\end{equation*}
In particular, for any $\alpha_1, ..., \alpha_k$ and any $\varepsilon > 0$ it is possible to find an integer $q$ such that $\| q\alpha_i \|$ is less than $\varepsilon$ for every $i = 1, \dots, k$.

\begin{refproof}[Proof of Theorem \ref{thm:countably_many_alpha}.]
	Let $\mathcal{A} = \{\alpha_1, \alpha_2, \dots\}$ with $\alpha_i$, $i \in \N$ irrational. We show that there exists a sequence $(a_n)_{n \geq 1}$ such that for all $i = 1, 2, \dots$ and all $k_1,k_2 \in \N_0$ with $\{k_2\alpha_i\} > \{k_1\alpha_i\}$, the interval
	\begin{equation*}
	J_i = [\{k_1\alpha_i\},\{k_2\alpha_i\}),
	\end{equation*}
	is a BRS for the sequence $(\{a_n\alpha_i\})_{n \geq 1}$.
	
	The minimal pairwise distance between any two of the first $k$ elements of $(\{n\alpha\})_{n \geq 1}$ is certainly at least $\min_{1 \leq l \leq k} \| l\alpha \|$.
	Since we now have to deal with countably many $\alpha_i$ we define 
	\begin{equation*}
	\varepsilon_k := \min_{1 \leq n,i \leq k} \| n\alpha_i \|.
	\end{equation*}
	
	From the simultaneous version of Dirichlet's Approximation Theorem, for all of those $\varepsilon_k$ there exists a $q_k$ such that 
	\begin{equation*}
	\| \{q_k\alpha_i\} \| < \varepsilon_k
	\end{equation*}
	for all $i  = 1,...,k$.
	Since $k$ was arbitrary, this approach defines an increasing integer sequence $(q_k)_{k \geq 1}$.
	It can be easily shown that for any $i \in \N$, intervals of the form $J = [\{k_1\alpha_i\},\{k_2\alpha_i\})$ with $k_1, k_2 \in \N$ are bounded remainder sets for the sequence $(\{(q_n + n)\alpha_i\})_{n \geq 1}$:
	
	Consider again the sequences $(\{n\alpha_i\})_{n \geq 1}$ and $(\{(q_n + n)\alpha_i\})_{n \geq 1}$. Then, for all $k \geq \max(k_1 +1,k_2 + 1,i)$ it holds that the boundaries of $J$ are contained in the first $k$ elements of $(\{n\alpha\})$. Moreover, the minimal distance of the element $\{k\alpha_i\}$ to the boundaries of $J$, which is greater than $\varepsilon_k$, is larger than the distance between $\{k\alpha_i\}$ and $\{(q_k +k)\alpha_i\}$, which is given by $\| q_k\alpha_i \|$. Therefore, it holds for every $i \in \N$ that
	\begin{equation*}
	\#\{K \leq n \leq N: \{a_n\alpha_i\} \in J\}
	= \#\{K \leq n \leq N: \{n\alpha_i\} \in J\},
	\end{equation*}
	for all $N \in \N$ where $K = \max(k_1 +1,k_2 + 1,i)$. This implies that  
	\begin{equation*}
	|\#\{1 \leq n \leq N: \{a_n\alpha_i\} \in J\} - \lambda(J)N| \leq c
	\end{equation*}
	for all $N \in \N$ and with a constant $c$ independent of $N$. The intervals $J$ form a dense collection of BRS such that by Lemma \ref{lma:dcbrs_ud} the sequences $(\{a_n\alpha_i\})_{n \geq 1}$ are uniformly distributed.
\end{refproof}

%
%
%
%
%
%

\section{Proof of Theorem \ref{thm:sn-brs_polynomial}}
\label{sec:proof_2}

This section is dedicated to the proof of Theorem \ref{thm:sn-brs_polynomial} for which two auxiliary results are needed. The proofs of the lemmas as well as of the theorem basically follow the same idea, nevertheless, it is necessary to perform them separately.

\begin{lemma}
	\label{thm:different_alpha_sequence}
	Let $q \in \N$, $\alpha_1, \alpha_2, \dots, \alpha_q$ be irrational and linearly independent over $\mathbb{Q}$. Let $(\bm{x}_n)\lc = (\{n\alpha_1\}, \{n \alpha_2\}, \dots, \{n\alpha_q\})\lc$.
	Then, for all intervals $[\bm{0},\bm{b}) = [0,b_1)\times [0, b_2) \times \dots \times [0, b_q) \subset [0,1)^q$ of positive measure there exists $K \in \N$ such that at least one of $K$ consecutive elements of $(\bm{x}_n)\lc$ is contained in $[\bm{0},\bm{b})$.
\end{lemma}

\begin{proof}
	Let $(\bm{x}_n)\lc = (\{n\alpha_1\}, \{n \alpha_2\}, \dots, \{n\alpha_q\})\lc$, where $\alpha_1, \dots, \alpha_q$ are irrational and linearly independent over $\mathbb{Q}$.
	To begin with, we partition the interval $[0,1)^q$ into $l^q$ subintervals of volume $l^{-q}$, i.e. intervals of the form
	\begin{equation*}
	\mathcal{I}(j_1, \dots, j_q) = \left[\frac{j_1}{l},\frac{j_1+1}{l}\right) \times \dots \times \left[\frac{j_q}{l},\frac{j_q+1}{l}\right),
	\end{equation*}
	with $j_1, \dots, j_q \in \{0, 1, \dots, l-1\}$ where $l$ is chosen such that 
	\begin{equation}
	\label{eq:supset_b}
	[\bm{0}, \bm{b}) \supseteq \left[0,\frac{2}{l} \right)^q.
	\end{equation}
	
	Since $(\bm{x}_n)\lc$ is uniformly distributed, each interval of the form $\mathcal{I}(j_1, \dots, j_q)$ contains elements of $(\bm{x}_n)\lc$. For each $\bm{j} = (j_1, \dots, j_q) \in \{0, 1, \dots, l-1\}^q$, by $K_{\bm{j}}$ we denote the index of the first element of $(\bm{x}_n)\lc$ which lies in $\mathcal{I}(\bm{j})$.
	Now define
	\begin{equation*}
	K_{\bm{\alpha},\bm{b}} =\max_{\bm{j} \in \{0,1, \dots, l-1\}^q}K_{\bm{j}}.
	\end{equation*}
	Then, assume that for any $N \in \N$ we have 
	\begin{equation}
	\label{eq:elements}
	x_{N+k} \notin [\bm{0},\bm{b})
	\end{equation}
	for all $k = 1, \dots, K_{\bm{\alpha},\bm{b}}$. Note that if $c$ is a constant, it holds that $\{x + c\} \in [a,b)$ is equivalent to $\{x \} \in [a - c, b - c) \pmod 1$ where we use the notation
	\begin{equation*}
	[a,b) \pmod 1 :=
	\begin{cases}
	[a \pmod 1, b \pmod 1) \quad &\text{if } a \pmod 1 > b \pmod 1 \\
	[a \pmod 1, 1) \cup [0, b \pmod 1) \quad &\text{if } a \pmod 1 < b \pmod 1.
	\end{cases}
	\end{equation*}
	Therefore, $\bm{x}_{N + k} \notin [\bm{0}, \bm{b})$ implies that
	\begin{equation*}
	\bm{x}_k \notin [-\bm{x}_N, \bm{b} -\bm{x}_N) \pmod 1
	\end{equation*}
	for all $k = 1, \dots, K_{\bm{\alpha},\bm{b}}$. 
	Because of \ref{eq:supset_b} the interval $[-\bm{x}_N, \bm{b} -\bm{x}_N) \pmod 1$ fully covers at least one interval of the form $\mathcal{I}(j_1, \dots, j_q)$ which has to contain at least one of the first $K_{\bm{\alpha},\bm{b}}$ points of $(\bm{x}_n)\lc$. Thus, we have a contradiction to (\ref{eq:elements}).
\end{proof}

\begin{lemma}
	\label{lma:combination}
	Let $p,q \in \N$, $\alpha_1, \alpha_2, \dots, \alpha_q$ be irrational and linearly independent over $\mathbb{Q}$. Let $(\bm{x}_n)\lc$ be the $pq$-dimensional sequence whose components have the form $\{n^{p'}\alpha_{q'}\}$ with $1 \leq p' \leq p$ and $1 \leq q' \leq q$.
	Then, for all intervals $[\bm{0},\bm{b}) \subset [0,1)^{pq}$ there exists $K \in \N$ such that at least one of $K$ consecutive elements of $(\bm{x}_n)\lc$ is contained in $[\bm{0},\bm{b})$.
\end{lemma}

\begin{proof}
	The elements of the sequence $(\bm{x}_n)\lc$ are given by
	\begin{equation*}
	\bm{x}_n = (\{n\alpha_1\}, \{n^2\alpha_1\}, \dots, \{n^p\alpha_1\}, \dots, \{n\alpha_q\}, \{n^2\alpha_q\}, \dots, \{n^p\alpha_q\}).
	\end{equation*}
	We prove the result by induction on $p$. By Lemma \ref{thm:different_alpha_sequence} the assertion holds for $p = 1$ and arbitrary $q \in \N$, which proves the base case.
	
	Now suppose that for all $\bm{b}' \in [0,1)^{(p-1)q}$ there exists $L \in \N$ such that at least one of $L$ consecutive elements of $(\bm{x}_n')\lc$, where
	\begin{equation*}
	\bm{x}'_n = (\{n\alpha_1\}, \{n^2\alpha_1\}, \dots, \{n^{p-1}\alpha_1\}, \dots, \{n\alpha_q\}, \{n^2\alpha_q\}, \dots, \{n^{p-1}\alpha_q\})
	\end{equation*} 
	is contained in $[\bm{0},\bm{b}')$.
	For the inductive step, we have to show that for all $\bm{b} \in [0,1)^{pq}$ there exists $K \in \N$ such that at least one of $K$ consecutive elements of $(\bm{x}_n)\lc$ lies in  $[\bm{0},\bm{b})$.
	
	We start by partitioning the interval $[0,1)^{pq}$ into $l^{pq}$ subintervals of volume $l^{-pq}$, i.e. intervals of the form
	\begin{equation*}
	\mathcal{I}(j_1^{(1)}, \dots, j_q^{(p)}) = \left[\frac{j_1^{(1)}}{l},\frac{j_1^{(1)}+1}{l}\right) \times \dots \times \left[\frac{j_q^{(p)}}{l},\frac{j_q^{(p)}+1}{l}\right),
	\end{equation*}
	with $j_{q'}^{(p')} \in \{0, 1, \dots, l-1\}$ for $q' \leq q$ and $p' \leq p$.
	Moreover, $l$ is chosen such that 
	\begin{equation*}
	b_1^{(1)}, \dots, b_q^{(p)} \geq \frac{2R + 2}{l}
	\end{equation*}
	with $R = \sum_{j = 1}^{p - 1}\binom{p}{j}$
	in order to guarantee that 
	\begin{equation}
	\label{eq:supset}
	[\bm{0}, \bm{b}) = [0,b_1^{(1)}) \times \dots \times [0,b_q^{(p)}) \supseteq \left[0,\frac{2R + 2}{l} \right)^{pq}.
	\end{equation}
	
	Since for each lattice point $\bm{h} \in \Z^{pq}, \bm{h} \neq 0$, the polynomial 
	\begin{equation*}
	\langle \bm{h}, (n\alpha_1, \dots, n^p\alpha_1, \dots, n\alpha_q, \dots, n^p\alpha_q) \rangle
	\end{equation*}
	has at least one non-constant term with irrational coefficient,
	the sequence $(\bm{x}_n)\lc$ is uniformly distributed (see e.g. Theorem 6.4 in \cite{Kuipers:Uniform_Distribution_of_Sequences}). Hence, each interval of the form $\mathcal{I}(j_1^{(1)}, \dots, j_q^{(p)})$ contains elements of $(\bm{x}_n)\lc$ and for each $\bm{j} = (j_1^{(1)}, \dots, j_q^{(p)}) \in \{0, 1, \dots, l-1\}^{pq}$, by $K_{\bm{j}}$ we denote the index of the first element of $(\bm{x}_n)\lc$ which lies in $\mathcal{I}(\bm{j})$.
	As before, define
	\begin{equation*}
	K_{\bm{\alpha},\bm{b}} = \max_{\bm{j} \in \{0,1, \dots, l-1\}^{pq}}K_{\bm{j}}
	\end{equation*}
	and note that $K_{\bm{\alpha},\bm{b}} > l^q$.
	By the induction hypothesis, the assertion of Lemma \ref{lma:combination} is true for $p -1$ and any $\bm{b}' \in [0,1)^{(p-1)q}$. We consider the special case $\bm{b}' = (K_{\alpha,\bm{b}}^{-p}, \dots, K_{\alpha,\bm{b}}^{-p}) \in \ui^{(p-1)q}$ and conclude that there exists $L = L(\alpha_1, \dots, \alpha_{q}, \bm{b})$ such that in any $L$ consecutive integers there is at least one integer $M$ with the property
	\begin{equation}
	\label{eq:2V_smaller_2}
	\|M^{p'}\alpha_{q'}\| \leq K_{\bm{\alpha},\bm{b}}^{-p},
	\end{equation}
	for all $q' \leq q$ and $p' \leq p-1$.
	Furthermore, define $K = L + K_{\bm{\alpha},\bm{b}}$ and and suppose that for some $N \in \N$ it holds that
	\begin{equation*}
	\label{eq:considered_points}
	\bm{x}_{n} \notin [\bm{0},\bm{b})
	\end{equation*}
	for $n = N+1, \dots, N+K$.
	More specifically, we consider the elements $\bm{x}_{M + k}$ for $k = 1, \dots, K_{\bm{\alpha},\bm{b}}$ where $N +1 \leq M \leq N + L$ and $M$ satisfies inequalities (\ref{eq:2V_smaller_2}).
	For the components of these elements with $p' = 1$ it holds that
	\begin{equation*}
	\{(M+k)\alpha_{q'}\} = \{\underbrace{M\alpha_{q'}}_{=: \xi_{q'}^{(1)}} + k\alpha_{q'}\} \notin [0,b_{q'}^{(1)}),
	\end{equation*}
	for all $q' \leq q$.
	This is equivalent to
	\begin{equation*}
	\{k\alpha_{q'}\} \notin [- \xi_{q'}^{(1)},b_{q'}^{(1)} - \xi_{q'}^{(1)}) \pmod 1.
	\end{equation*}
	For the components with $p' = 2$ we have
	\begin{equation*}
	\{(M+k)^2\alpha_{q'}\} = \{\underbrace{M^2\alpha_{q'}}_{=:\xi_{q'}^{(2)}} + 2Mk\alpha_{q'} + k^2\alpha_{q'}\} \notin [0,b_{q'}^{(2)}).
	\end{equation*}
	Since $\{2Mk\alpha_{q'} + k^2\alpha_{q'}\}$ is either $\{k^2\alpha_{q'}\} + \|2Mk\alpha_{q'}\|$ or $\{k^2\alpha_{q'}\} - \|2Mk\alpha_{q'}\|$ and with the estimate 
	\begin{align*}
	\|2Mk\alpha_{q'}\| &\leq \|M\alpha_{q'}\|2k \\
	& \leq \|M\alpha_{q'}\|2K_{\bm{\alpha},\bm{b}} \\
	& \leq \frac{2K_{\bm{\alpha},\bm{b}}}{K_{\bm{\alpha},\bm{b}}^p} < \frac{2}{l},
	\end{align*}
	we can conclude that
	\begin{equation*}
	\{k^2\alpha_{q'}\} \notin \left[- \xi_{q'}^{(2)} + \frac{2}{l}, b_{q'}^{(2)} - \xi_{q'}^{(2)} - \frac{2}{l}\right) \pmod 1.
	\end{equation*}
	We continue in this manner for all values of $p'$ until we arrive at
	\begin{equation*}
	\{(M+k)^p\alpha_{q'}\} = \{\underbrace{M^p\alpha_{q'}}_{=:\xi_{q'}^{(p)}} + \sum_{j = 1}^{p -1}\binom{p}{j}M^{p-j}k^j\alpha_{q'} + k^p\alpha_{q'}\} \notin [b_{q'}^{(p)},1)
	\end{equation*}
	for all $q' \leq q$.
	We apply the estimate
	\begin{align*}
	\|\binom{p}{j}M^{p-j}k^j\alpha_{q'}\| &\leq \|M^{p-j}\alpha_{q'}\|\binom{p}{j}k^j \\
	& \leq \|M^{p-j}\alpha_{q'}\|\binom{p}{j}K_{\bm{\alpha},\bm{b}}^{p-1} \\
	& \leq \binom{p}{j}K_{\bm{\alpha},\bm{b}}^{-1} \leq \binom{p}{j}l^{-1},
	\end{align*}
	for all $j = 1, \dots, p - 1$. Hence,
	\begin{equation*}
	\{k^p\alpha_{q'}\} \notin \left[- \xi_{q'}^{(p)} + \sum_{j = 1}^{p -1}\binom{p}{j}l^{-1},
	b_{q'}^{(p)} - \xi_{q'}^{(p)} - \sum_{j = 1}^{p -1}\binom{p}{j}l^{-1}\right) \pmod 1.
	\end{equation*}
	
	Finally, we use (\ref{eq:supset}) and obtain
	\begin{align*}
	\{k\alpha_{q'}\} &\notin [- \xi_{q'}^{(1)}, (2R+2)l^{-1} - \xi_{q'}^{(1)})  \pmod 1 =: \mathcal{I}_{q'}^{(1)}, \\
	\{k^2\alpha_{q'}\} &\notin \left[- \xi_{q'}^{(2)} + 2l^{-1}, (2R+2)l^{-1} - \xi_{q'}^{(2)} -2l^{-1}\right)  \pmod 1 =: \mathcal{I}_{q'}^{(2)}, \\
	&\vdots\\
	\{k^p\alpha_{q'}\} &\notin \left[- \xi_{q'}^{(p)} + \sum_{j = 1}^{p -1}\binom{p}{j}l^{-1},
	(2R+2)l^{-1} - \xi_{q'}^{(p)} - \sum_{j = 1}^{p -1}\binom{p}{j}l^{-1}\right) \pmod 1 =: \mathcal{I}_{q'}^{(p)}
	\end{align*}
	for all $k = 1, \dots, K_{\bm{\alpha},\bm{b}}$ and all $q' \leq q$.
	However, because of the definition of $R$ each of these one-dimensional intervals is of length greater or equal to $2/l$. The $pq$-dimensional interval $\mathcal{I}_1^{(1)} \times \dots \times \mathcal{I}_q^{(p)}$ therefore fully contains at least one of the intervals $\mathcal{I}(\bm{j})$. Since $K_{\bm{\alpha},\bm{b}}$ was constructed such that each interval $\mathcal{I}(\bm{j})$ at least contains one of the elements $(\bm{x}_n)_{1 \leq n \leq K_{\bm{\alpha},\bm{b}}}$ we have a contradiction. Therefore, for each $N$ at least one element of $(\bm{x}_{n})_{N+1 \leq n \leq N + K}$ lies in $[\bm{0},\bm{b})$.
\end{proof}

Finally, we are able to give a proof of Theorem \ref{thm:sn-brs_polynomial}. We will use the notation $\{(x^{(1)}, \dots, x^{(s)})\} = (\{x^{(1)}\}, \dots, \{x^{(s)}\})$.

\begin{refproof}[Proof of Theorem \ref{thm:sn-brs_polynomial}.]
	Let $\bm{p}(n)$ fulfill the assumptions of the theorem and let $q$ be the highest degree of the polynomials $p^{(i)}(n)$, $i = 1, \dots, s$. Then, each of these polynomials is of the form $p^{(i)}(n) = \alpha_0^{(i)} + n\alpha_1^{(i)} + \dots + n^q\alpha_q^{(i)}$ with coefficients $\alpha_{q'}^{(i)} \in \R$, $q' \leq q$. Furthermore, we split the set of coefficients into three subsets:
	\begin{align*}
	A_1 &= \{\alpha_{q'}^{(i)}: \alpha_{q'}^{(i)} \in \mathbb{Q}, i = 1, \dots, s, q' \leq q\},\\
	A_2 &= \{\alpha_{q'}^{(i)}: \alpha_{q'}^{(i)} \in \R \setminus \mathbb{Q}, i = 1, \dots, s, q' \leq q \text{ and all } \alpha_{q'}^{(i)} \text{ are linearly independent over } \mathbb{Q} \},\\
	A_3 &= \{\alpha_{q'}^{(i)}: \alpha_{q'}^{(i)} \in \R \setminus \mathbb{Q}, i = 1, \dots, s, q' \leq q \text{ and } \alpha_{q'}^{(i)} = \sum_{\alpha \in A_2} c(\alpha_{q'}^{(i)}, \alpha)\alpha\}.
	\end{align*}
	Note that this partition is not unique. Using these sets we can now define for each single $i$:
	\begin{align*}
	Q_1^{(i)} &= \{0 < q_1 \leq q: \alpha_{q_1}^{(i)} \in A_1 \}, \\
	Q_2^{(i)} &= \{0 < q_2 \leq q: \alpha_{q_2}^{(i)} \in A_2 \}, \\
	Q_3^{(i)} &= \{0 < q_3 \leq q: \alpha_{q_3}^{(i)} \in A_3 \},
	\end{align*}
	for all $i = 1, \dots, s$.
	For each $i$, the set $Q_2^{(i)} \cup Q_3^{(i)}$ contains at least one element. We have 
	\begin{equation*}
	p^{(i)}(n) = \underbrace{\sum_{q_1 \in Q_1^{(i)}} n^{q_1}\alpha_{q_1}^{(i)}}_{=: P_1^{(i)}} + 
	\underbrace{\sum_{q_2 \in Q_2^{(i)}} n^{q_2}\alpha_{q_2}^{(i)}}_{=: P_2^{(i)}} + 
	\underbrace{\sum_{q_3 \in Q_3^{(i)}} n^{q_3}\alpha_{q_3}^{(i)}}_{=: P_3^{(i)}} + \alpha_0^{(i)}.
	\end{equation*}
	First, we consider some interval $\abm \subset \ui^s$ and show that the maximal number of successive elements of $(\bm{p}(n))\lc$ which are not contained in $\abm$ is bounded.
	Therefore, let $D$ denote the least common multiple of the denominators of all $\alpha \in A_1$. We partition $[0,1)^s$ into intervals
	\begin{equation*}
	\mathcal{I}(\bm{j}) = \left[\frac{j_1}{l},\frac{j_1+1}{l}\right) \times \dots \times \left[\frac{j_s}{l},\frac{j_s+1}{l}\right),
	\end{equation*}
	with $\bm{j} = (j_1, \dots, j_s) \in \{0, 1, \dots, l-1\}^s$. The parameter $l$ is chosen such that 
	\begin{equation*}
	b^{(i)} - a^{(i)} \geq \frac{2 R^{(i)} + 2}{l},
	\end{equation*}
	for $i = 1, \dots, s$ in order to obtain
	\begin{equation}
	\label{eq:interval_contained}
	\abm \supseteq \left[\bm{a}, \bm{a} + \frac{2 R^{(i)} + 2}{l}\right).
	\end{equation}
	The constants $R^{(i)}$ are defined as
	\begin{equation*}
	R^{(i)} := \sum_{q_2 \in Q_2^{(i)}}\sum_{j = 1}^{q_2-1}\binom{q_2}{j}D^j + 
	\sum_{q_3 \in Q_3^{(i)}}\sum_{j = 1}^{q_3-1}\binom{q_3}{j}D^j\sum_{\alpha \in A_2}c(\alpha_{q_3}^{(i)},\alpha).
	\end{equation*}
	Consider now the sequence $(\{P_2^{(1)}(Dn) + P_3^{(1)}(Dn)\}, \dots, \{P_2^{(s)}(Dn) + P_3^{(s)}(Dn)\})\lc$ which, by the assumption of the theorem, is uniformly distributed. Therefore, each of the intervals $\mathcal{I}(\bm{j})$ contains points of the sequence and by $K_{\bm{j}}$ we denote the index of the first element in the respective interval. As before, we define 
	\begin{equation*}
	K_{\bm{\alpha}, \bm{b}} = \max_{\bm{j} \in \{0,1,\dots,l-1\}^s}K_{\bm{j}}.
	\end{equation*}
	Now consider the sequence whose components are given by the sequences $n^{q'}\alpha$ with $\alpha \in A_2$ and $q' \leq q$. Since all $\alpha \in A_2$ are linearly independent, it follows from Lemma \ref{lma:combination} that there exists $L \in \N$ such that any $L$ integers contain at least one integer $M$ which fulfills
	\begin{equation}
	\label{eq:inequality_p(n)}
	\| M^{q'}\alpha\| < K_{\bm{\alpha}, \bm{b}}^{-q}
	\end{equation}
	for all $q' \leq q$, $\alpha \in A_2$.
	We use this $L$ to define $K = L + DK_{\bm{\alpha},\bm{b}}$ and show that for each $N \in \N$ at least one of the elements 
	\begin{equation*}
	\{\bm{p}(N + 1)\}, \dots, \{\bm{p}(N + K)\}
	\end{equation*}
	lies in $\abm$, i.e. the number of consecutive elements of $(\bm{p}(n))\lc$ which are not contained in $\abm$ is bounded by $K$.
	The elements $\{\bm{p}(N + 1)\}, \dots, \{\bm{p}(N + L)\}$ contain an element $\{\bm{p}(M)\}$, where $M$ satisfies (\ref{eq:inequality_p(n)}). Therefore, we particularly consider the elements $\{\bm{p}(M + Dk)\}$ for $k = 1, \dots, K_{\bm{\alpha},\bm{b}}$. Suppose that
	\begin{equation}
	\label{eq:assertion}
	\{\bm{p}(M + Dk)\} \notin \abm
	\end{equation}
	for all $k = 1, \dots, K_{\bm{\alpha},\bm{b}}$ and use
	\begin{equation*}
	p^{(i)}(M + Dk)
	= P_1^{(i)}(M + Dk) + P_2^{(i)}(M + Dk) + P_3^{(i)}(M + Dk) + \alpha_0^{(i)}.
	\end{equation*}
	Since $D$ is the least common multiple of the denominators of $\alpha \in A_1$, it holds that
	\begin{equation*}
	P_1^{(i)}(M + Dk) = P_1^{(i)}(M) + z
	\end{equation*}
	for some $z \in \Z$.
	For $P_2^{(i)}$ we have
	\begin{align*}
	P_2^{(i)}(M + Dk) 
	&= \sum_{q_2 \in Q_2^{(i)}}(M + Dk)^{q_2}\alpha_{q_2}^{(i)} \\
	&= \underbrace{\sum_{q_2 \in Q_2^{(i)}} M^{q_2}\alpha_{q_2}^{(i)}}_{ = P_2^{(i)}(M)} + 
	\sum_{q_2 \in Q_2^{(i)}}\sum_{j = 1}^{q_2-1}\binom{q_2}{j} M^{q_2-j}(Dk)^j\alpha_{q_2}^{(i)} + 
	\underbrace{\sum_{q_2 \in Q_2^{(i)}}(Dk)^{q_2}\alpha_{q_2}^{(i)}}_{= P_2^{(i)}(Dk)}.
	\end{align*}
	Similarly, for $P_3^{(i)}$ we obtain
	\begin{align*}
	P_3^{(i)}(M + Dk) 
	&= \underbrace{\sum_{q_3 \in Q_3^{(i)}} M^{q_3}\alpha_{q_3}^{(i)}}_{ = P_3^{(i)}(M)} + 
	\sum_{q_3 \in Q_3^{(i)}}\sum_{j = 1}^{q_3-1}\binom{q_3}{j} M^{q_3-j}(Dk)^j\alpha_{q_3}^{(i)} + 
	\underbrace{\sum_{q_3 \in Q_3^{(i)}}(Dk)^{q_3}\alpha_{q_3}^{(i)}}_{= P_3^{(i)}(Dk)},
	\end{align*}
	where the second term can be written as
	\begin{equation*}
	\sum_{q_3 \in Q_3^{(i)}}\sum_{j = 1}^{q_3-1}\binom{q_3}{j} \sum_{\alpha \in A_2} c(\alpha_{q_3}^{(i)},\alpha)M^{q_3-j}(Dk)^j\alpha.
	\end{equation*}
	By the inequalities in (\ref{eq:inequality_p(n)}) it holds that
	\begin{equation*}
	\| M^{q'-j}k^j\alpha \| \leq \| M^{q'-j}\alpha \| K_{\alpha,\bm{b}}^{q} < l^{-1},
	\end{equation*}
	for all $q' \leq q$, $j \leq q'$ and $\alpha \in A_2$.
	Using the fact that $\{x\} - c\| y \| \leq  \{ x + cy \} \leq \{x\} + c\| y \|$, where $c$ is a constant,
	the definition of $R^{(i)}$, (\ref{eq:interval_contained}) and (\ref{eq:assertion}), we obtain with $\xi^{(i)} = P_1^{(i)}(M) + P_2^{(i)}(M) + P_3^{(i)}(M) + \alpha_0^{(i)}$  that
	\begin{equation*}
	\{P_2^{(i)}(Dk) + P_3^{(i)}(Dk) \} \notin [a^{(i)} - \xi^{(i)} + R^{(i)}l^{-1}, a^{(i)} + (2R^{(i)} + 2)l^{-1} - \xi^{(i)} - R^{(i)}l^{-1}) \pmod 1
	\end{equation*}
	for all $ i = 1, \dots s$ and all $k = 1, \dots, K_{\bm{\alpha},\bm{b}}$.
	The length of each of those intervals is $2/l$, therefore, the resulting $s$-dimensional interval covers at least one of the intervals $\mathcal{I}(\bm{j})$. Since $K_{\bm{\alpha},\bm{b}}$ was chosen such that any $\mathcal{I}(\bm{j})$ contains at least one element of 
	$(\{P_2^{(1)}(Dk) + P_3^{(1)}(Dk)\}, \dots, \{P_2^{(s)}(Dk) + P_3^{(s)}(Dk)\})_{1\leq k \leq K_{\bm{\alpha},\bm{b}}}$ this is a contradiction to the construction of $K_{\bm{\alpha},\bm{b}}$.
	Thus, the maximal number of successive elements of $(\{\bm{p}(n)\})\lc$ which are not contained in some interval $\abm \subset \ui^s$ is bounded by some constant $K$. Obviously, the same holds for the number of successive elements of $(\{\bm{p}(n)\})\lc$ in $\ui^s \setminus \abm$.
	
	In order to prove that the sequence $(\{\bm{p}(n)\})\lc$ does not have any S-NBRS, it remains to show that the number of successive elements in $\abm \subset \ui^s$ is also bounded. This can be easily seen since either
	\begin{equation*}
	\abm \subseteq [\bm{0},\bm{b}) \subset \ui^s \setminus [\bm{b},\bm{1}),
	\end{equation*}
	or
	\begin{equation*}
	\abm \subseteq [\bm{a},\bm{1}) \subset \ui^s \setminus [\bm{0},\bm{a}).
	\end{equation*}
	We already know that the number of consecutive elements in $\ui^s \setminus [\bm{b},\bm{1})$ as well as in $\ui^s \setminus [\bm{0},\bm{a})$ is bounded, therefore this also holds for the interval $\abm$.
\end{refproof}

%
%
%
%
%
%

\bibliographystyle{plain}
\bibliography{literatur}

\textbf{Author’s Addresses:} \\ 
Lisa Kaltenböck and Gerhard Larcher, Institut für Finanzmathematik und Angewandte Zahlentheorie, Johannes Kepler Universität Linz, Altenbergerstraße 69, A-4040 Linz, Austria. \\ \\ 
Email: lisa.kaltenboeck(at)jku.at, gerhard.larcher(at)jku.at 
\end{document}